\documentclass[12pt]{amsart}

\usepackage{amsfonts}
\usepackage{amsmath}
\usepackage{amssymb}
\usepackage{amstext}
\usepackage{amsthm}
\usepackage{mathtools}
\usepackage{amsopn}
\usepackage{mathrsfs}
\usepackage{tikz-cd}
\usepackage{geometry}
\usepackage{url}
\usepackage{hyperref}
\usepackage{amsrefs}

\hypersetup{
	colorlinks=true,
	linkcolor=red,
	filecolor=magenta,      
	urlcolor=cyan,
}

\theoremstyle{definition}
\theoremstyle{remark}
\newtheorem{theorem}{Theorem}
\newtheorem{lemma}{Lemma}[section]
\newtheorem{proposition}[lemma]{Proposition}
\newtheorem{corollary}[lemma]{Corollary}
\newtheorem{definition}[lemma]{Definition}
\newtheorem{remark}{Remark}[section]
\newtheorem{example}[remark]{Example}
\newtheorem{notation}[remark]{Notation}
\newtheorem*{theorem1}{Theorem A}
\numberwithin{equation}{section}

\DeclareMathOperator{\lm}{lm}
\DeclareMathOperator{\lc}{lc}
\DeclareMathOperator{\lcm}{lcm}
\DeclareMathOperator{\lt}{lt}
\DeclareMathOperator{\homology}{H}

\begin{document}

\title{Gr\"{o}bner bases and equations of the multi-Rees algebras}

\author{Babak Jabarnejad}
\address{Department of Mathematical sciences, University of Arkansas, Fayetteville, Arkansas, 72701, USA}
\email{babak.jab@gmail.com}


\date{}

\dedicatory{}

\begin{abstract}
In this paper we describe the equations defining the multi-Rees algebra $R[I_1^{a_1}t_1,\dots,I_r^{a_r}t_r]$, where $R$ is a Noetherian ring and the ideals are generated by subsets of a fixed weak regular sequence.
\end{abstract}

\maketitle

\section{Introduction}

Let $R$ be a Noetherian ring, and let $s_1,\dots,s_n$ be generators of the ideal $I$. We define the homomorphism $\phi$ from the
polynomial ring $S=R[T_1,\dots,T_n]$ to the Rees algebra $R[It]$ by sending $T_i$ to $s_it$. Then $R[It]\cong S/\ker(\phi)$. The generating set of $\ker(\phi)$ is referred to as the
defining equations of the Rees algebra $R[It]$. Finding these generating sets is a tough problem which is open for most classes of ideals. Some papers about this problem are \cite{vasconcelos1991rees}, \cite{johnson2006equations}, \cite{vasconcelosulrich1993rees}, \cite{morey1996rees}, \cite{moreyulrich1996rees}, \cite{kustinpoliniulrich2017blowup}.

More generally, given any ideals $I_1,\dots,I_r$ in a ring $R$, one would like to describe the
equations of the multi-Rees algebra $R[I_1t_1,I_2t_2,\dots,I_rt_r]$. Indeed, the multi-Rees algebra in question is simply the Rees algebra of the module $I_1\oplus I_2\oplus\dots\oplus I_r$. However, in our work, we make no serious use of this theory. There is little work on the defining equations of the multi-Rees algebra compared to the ordinary Rees algebra. Another motivation for investigating the multi-Rees algebra is an illustration of the theory of Rees algebra of modules \cite{eisenbud2003rees}, \cite{simis2003rees}. An important paper of Ribbe~\cite{ribbe1999defining}, describes the equations of the multi-Rees algebra $R[It_1,\dots,It_r]$, when $I$ is an ideal of linear type. In \cite{jabarnejad2016rees}, we determine the equations of the multi-Rees algebra $R[I^{a_1}t_1,\dots,I^{a_r}t_r]$ for any ideal $I$ of linear type.

The concept of Gr\"{o}bner bases for polynomial rings over a field was presented by Buchberger~\cite{buchberger1965algorithmus}. He also gave generalizations of this concept over some rings (e.g.~\cite{buchberger1984critical}). One may see a generalization of this concept over arbitrary rings in \cite{adams1994introduction}.

In this paper we determine the equations of the multi-Rees algebra $R[I_1^{a_1}t_1,I_2^{a_2}t_2,\dots,I_r^{a_r}t_r]$, where $R$ is any Noetherian ring and ideals $I_i$ are generated by subsets of a fixed permutable weak regular sequence. In general, these equations can have arbitrarily large degrees. First, we give a slight generalization of Buchberger's Criterion, then applying this generalized notion we prove the main result for $R[I_1t_1,I_2t_2,\dots,I_rt_r]$. Next, we use the method given in~\cite{jabarnejad2016rees} to prove the main result for $R[I_1^{a_1}t_1,I_2^{a_2}t_2,\dots,I_r^{a_r}t_r]$.

To describe the equations, we introduce the notion of a quasi-matrix and that of a binary quasi-minor, which serves as a generalization $2\times 2$-minors. 

\begin{theorem1}
Let R be a Noetherian ring and suppose that ideals $I_i$ are generated by subsets of a fixed permutable weak regular sequence $s_1,\dots,s_n$. Then there is a quasi-matrix $D$, whose entries are certain indeterminates, such that the multi-Rees algebra $R[I_1^{a_1}t_1,\dots,I_r^{a_r}t_r]$ is defined by the ideal generated by all binary quasi-minors of $[\underline{s}|D]$.
\end{theorem1}

For example in the ring $R=\mathbb{Z}[x,y,z]$, we consider permutable weak regular sequence $2,3,5,x,y,z$. The ideals $I_1=\langle 2,x,y\rangle$, $I_2=\langle 3,x,z\rangle$ and $I_3=\langle 5,y,z\rangle$ are not equal to $R$ and describing the equations of the multi-Rees algebra $R[I_1t_1,I_2t_2,I_3t_3]$ is an interesting question. For another example, 
if $R$ is a Noetherian ring and $a,b,c$ are nonzero divisors, then on the ring $R\times R\times R[x,y,z]$ we have permutable weak regular sequence $(a,1,1),(1,b,1),(1,1,c),x,y,z$. Another example could be a regular sequence of homogeneous elements with positive degree in the Noetherian ring $R[x_1,\dots,x_n]$.

\section{Gr\"{o}bner bases on $\underline{s}$-monomial type polynomials}

Let $R$ be a commutative ring with unit. We fix $S=R[x_1,...,x_n]$. Assume we have a monomial order $>$ on $S$. If $f\in S$, then we denote the leading term of $f$ by $\lt(f)$, leading coefficient by $\lc(f)$ and leading monomial by $\lm(f)$.

\begin{definition}
Let $G=\{f_1,...,f_m\}\subseteq S$. We say $f\in S$ reduces to zero modulo $G$ and denote this by $f \rightarrow_{G} 0$ if there are $p_i\in S$, for $1\le i \le m \ $, such that
$$
f=p_1f_1+...+p_mf_m,\ \text{and} \ \lm(f) \ge \lm(p_i)\lm(f_i)  
$$
\end{definition}

\begin{definition}
Let $I$ be an ideal of $S$ and $\{f_1,...,f_m\}\subseteq I$. We say $\{f_1,...,f_m\}$ is a Gr\"{o}bner basis for $I$ if $\langle \lt(f_1),...,\lt(f_m)\rangle=\lt(I)$, where $\lt(I)$ is the ideal generated by leading terms of elements of $I$.
\end{definition}

If $R$ is a Noetherian ring, then the Gr\"{o}bner basis of $I$ always exists.

As we know the sequence $\underline{s}=s_1,\dots,s_k$ is a weak regular sequence if $s_i$ is an $R/\langle s_1,\dots,s_{i-1}\rangle$-regular element for $i=1,\dots,k$. We fix a permutable weak regular sequence $\underline{s}=s_1,\dots,s_k$ in $R$, where none of them is a unit. By an $\underline{s}$-monomial we mean a monomial in this fixed weak regular sequence and by an $\underline{s}$-term we mean an element of the form $ua$, where $u$ is a unit and $a$ is an $\underline{s}$-monomial. 

\begin{remark}
Any $\underline{s}$-monomial has a unique representation. Also, we consider 1 as an $\underline{s}$-monomial even
though 1 is also a monomial in $S$.
\end{remark}

\begin{remark}
If $a,b\in R$ and $a,b$ are $\underline{s}$-terms, then $\lcm(a,b)$ and $\gcd(a,b)$ exist and they are unique up to unit (by $\lcm$ (resp. $\gcd$) we mean the least common multiple (resp. the greatest common divisor)). If $a=u_1s_1^{\alpha_1}\dots s_k^{\alpha_k},\ b=u_2s_1^{\beta_1}\dots s_k^{\beta_k}$, $\alpha_i,\beta_i\in\mathbb{Z}_{\ge 0}$ are $\underline{s}$-terms, then we canonically choose $\gcd(a,b)=s_1^{\gamma_1}\dots s_k^{\gamma_k}$ and $\lcm(a,b)=s_1^{\eta_1}\dots s_k^{\eta_k}$, where $\gamma_i=\min\{\alpha_i,\beta_i\}$ and $\eta_i=\max\{\alpha_i,\beta_i\}$.
\end{remark}

The following lemma is similar to \cite[Lemma 3]{kaplansky1962r}.

\begin{lemma}\label{kaplansky}
Let $I$ be an ideal finitely generated by $\underline{s}$-monomials in $s_2,\dots,s_k$. Then $ts_1\in I$ implies $t\in I$.
\end{lemma}

The following lemma is similar to  \cite[Lemma 17]{taylor1966ideals} with a similar proof. Since Taylor's thesis cannot be found easily we provide the proof. 

\begin{lemma}\label{pre-syzygy}
Let $f,f_1,\dots,f_m\in R$ be $\underline{s}$-monomials. If $af=a_1f_1+\dots+a_mf_m$, then $a\in\langle g_1,\dots,g_m\rangle$, where for every $1\le i\le m$, $g_i=\frac{\lcm(f_i,f)}{f}=\frac{f_i}{\gcd(f_i,f)}$.
\end{lemma}
\begin{proof}
We prove the claim by induction on the total degree $d$ of $f$. If $d=1$ it is clear by Lemma~\ref{kaplansky}. We assume that claim is true for all monomials of total degree less than $d$. Without loss of generality we may assume that the exponent of $s_1$ in $f$ is nonzero. Let $f_1,\dots,f_t$, $t\le m$ be precisely those $f_i$ for which the exponent of $s_1$ is zero. Write $h_i=f_i$ for $1\le i\le t$ and for $t+1\le i\le m$, let $h_i$ be the $\underline{s}$-monomial obtained from $f_i$ by reducing the exponent of $s_1$ by 1. Also, let $h$ be
obtained from $f$ by reducing the exponent of $s_1$ by 1. We have $(ah-\sum_{i=t+1}^{m}a_ih_i)s_1=\sum_{i=1}^{t}a_ih_i$. Since $s_1$ does not occur in $h_1,\dots,h_t$, by the case $d=1$, we have $ah-\sum_{i=t+1}^{m}a_ih_i=\sum_{i=1}^{t}b_ih_i$. Now the total degree of $h$ is $d-1$ and hence by induction the claim proceeds. 
\end{proof}

At this stage we recall the concept of the Taylor complex of ideals generated by $\underline{s}$-monomials. 

Given an ideal $I=\langle a_1,\dots,a_n\rangle\subseteq R$, where the  $a_i$'s are $\underline{s}$-monomials, the Taylor complex $\mathbb{T}$ associated to these generators is a complex of $R$-modules as follows:

$$
\mathbb{T}: 0\rightarrow T_n\rightarrow T_{n-1}\rightarrow\dots\rightarrow T_2\rightarrow T_1\rightarrow T_0\rightarrow R/I\rightarrow 0.
$$

$T_0=R$, for $1\le p\le n$, $T_p$ is a free $R$-module of rank $\binom{n}{p}$ with basis $\{e_{i_1,\dots,i_p};\ 1\le i_1<i_2<\dots <i_p\le n\}$. The differential $d: T_p\rightarrow T_{p-1}$ is defined by
$$
d(e_{i_1,\dots,i_p})=\sum_{r=1}^{p}(-1)^{r-1}\frac{u(i_1,\dots,i_p)}{u(i_1,\dots,\hat{i}_r,\dots,i_p)}e_{i_1,\dots,\hat{i}_r,\dots,i_p},
$$ 
where $u(i_1,\dots,i_k)=\lcm(a_{i_1},\dots,a_{i_k})$ (for $p=1$, $d(e_i)=a_i$). We see easily $d\circ d=0$ and $\homology_{0}(\mathbb{T})=R/I$. Now Applying Lemma~\ref{pre-syzygy}, we have a similar lemma to \cite[Theorem 7.1.1]{herzog2011monomial} with similar proof for our case.

\begin{lemma}
Let $I=\langle a_1,\dots,a_n\rangle\subseteq R$ be an ideal and the generators be $\underline{s}$-monomials. Let $\mathbb{T}$ be the Taylor complex for this sequence. Then $\mathbb{T}$ is acyclic, and hence a graded free $R$-resolution of $R/I$. 
\end{lemma}

\begin{corollary}\label{syzygy}
Let $a_1,\dots,a_n\in R$ be $\underline{s}$-monomials. Then the  $R$-module $\{(c_1,\dots,c_n);\ c_1a_1+\dots+c_na_n=0\}\subseteq R^n$ is generated by $\{\frac{\lcm(a_i,a_j)}{a_i}e_i-\frac{\lcm(a_i,a_j)}{a_j}e_j;\ 1\le i < j\le n\}$.
\end{corollary}

\begin{definition}
If $f \in S$, and $\lc(f)$ is an $\underline{s}$-term, then we say $f$ is $\underline{s}$-monomial type.
\end{definition}

Let $f,g$ be $\underline{s}$-monomial type polynomials. Then we define:
$$
S(f,g)=\frac{\lcm(\lm(f),\lm(g))}{\lt(f)}f-\frac{\lcm(\lm(f),\lm(g))}{\lt(g)}g.
$$
It is clear that the $S$-ploynomial of such $f$ and $g$ is in $K[x_1,...,x_n]$, where $K$ is the total ring of fractions of $R$. Now, we define $S^{'}$-polynomial for $\underline{s}$-monomial type polynomials in $S$.

If $f,g \in S$ are $\underline{s}$-monomial type polynomials, then we define $S^{'}(f,g)$ by
$$
S^{'}(f,g)=\lcm(\lc(f),\lc(g))S(f,g).
$$ 
It is clear that $S^{'}(f,g)\in S$.

The following lemma is the key result that will be used to prove the Buchberger's Criterion for our case. But before we start we observe a notation: If $\delta=(\delta_1,\dots,\delta_n)\in \mathbb{Z}_{\ge 0}^{n}$, then by $x^{\delta}$ we mean $x_1^{\delta_1}\dots x_n^{\delta_n}$.  

\begin{lemma}\label{pre-main-lemma1}
Let $f_1,\dots,f_m\in S$ be $\underline{s}$-monomial type. Suppose we have a sum $\sum_{i=1}^{m}c_if_i$, where $c_i\in R$ and for all $i$, $\lm(f_i)=x^{\delta}, \delta\in \mathbb{Z}_{\ge 0}^{n}$. If $x^{\delta}>\lt(\sum_{i=1}^{m}c_if_i)$, then $\sum_{i=1}^{m}c_if_i$ is an $R$-linear combination of the $S^{'}$-polynomials $S^{'}(f_i,f_j)$, for $1\le i,j \le m$. We also have $x^{\delta}>\lm\left(S^{'}(f_i,f_j)\right)$.
\end{lemma}
\begin{proof}
If the leading coefficient of $f_i$ is $u_id_i$, where $u_i$ is a unit and $d_i$ is an $\underline{s}$-monomial, then we define $p_i=f_i/u_id_i$. Clearly $p_i$ is in the polynomial ring over the total ring of fractions of $R$. Since $x^{\delta}>\lt\left(\sum_{i=1}^{m}c_if_i\right)$, $c_1u_1d_1+...+c_mu_md_m=0$. We have
\begin{gather*}
\sum_{i=1}^{m} c_if_i=\sum_{i=1}^{m} c_iu_id_ip_i=\sum_{i=1}^{m-1}c_iu_id_i(p_i-p_m)=\sum_{i=1}^{m-1}c_iu_id_iS(f_i,f_m).
\end{gather*}
Applying Corollary~\ref{syzygy} we see that for $1\le i\le m-1$, we can write all $c_iu_i$'s as the following
\begin{gather*}
c_iu_i=\sum_{j=1}^{m}r_{i,j}\frac{\lcm(d_i,d_j)}{d_i}\Rightarrow c_iu_id_iS(f_{i},f_m)=\sum_{j=1}^{m}r_{i,j}\lcm(d_i,d_j)S(f_{i},f_m),
\end{gather*}
where $r_{k,k}=0$, and $r_{i,j}=-r_{j,i}$. Therefore we have 
\begin{gather*}
\sum_{i=1}^{m-1}c_iu_id_iS(f_{i},f_m)=\sum_{i=1}^{m-1}\sum_{j=1}^{m}r_{i,j}\lcm(d_i,d_j)S(f_{i},f_m).
\end{gather*}
In the second sum for arbitrary $i$ and $j=m$, we have $r_{i,m}\lcm(d_i,d_m)S(f_{i},f_m)=r_{i,m}S^{'}(f_{i},f_m)$. Also, for $1\le i<j\le m-1$ we have
\begin{gather*}
r_{i,j}\lcm(d_i,d_j)S(f_{i},f_m)+r_{j,i}\lcm(d_j,d_i)S(f_{j},f_m)\\
=r_{i,j}\lcm(d_i,d_j)S(f_{i},f_m)-r_{i,j}\lcm(d_i,d_j)S(f_{j},f_m)\\
=r_{i,j}\lcm(d_i,d_j)S(f_{i},f_{j})=r_{i,j}S^{'}(f_{i},f_{j}).
\end{gather*}
This completes the proof.
\end{proof}

Using Lemma~\ref{pre-main-lemma1} and \cite[Theorem 4.1.12]{adams1994introduction} with a similar proof to the proof of \cite[Theorem 6]{cox2007ideals} we have the following result.

\begin{theorem}[Buchberger's Criterion]\label{Grobner-basis}
Let $G=\{f_1,...,f_m\}$ be a family of $\underline{s}$-monomial type polynomials in $S$ and $I=\langle f_1,...,f_m\rangle$. Then $G$ is a Gr\"{o}bner basis for $I$ iff for every $i,j$, $S^{'}(f_i,f_j)\rightarrow_G 0$.
\end{theorem}

\section{Gr\"{o}bner basis and binary quasi-minors}

\begin{definition}
An $n\times m$ quasi-matrix over $R$ is a rectangular array with $n$ rows and $m$ columns such that some entries may be empty.
	
A subquasi-matrix is a quasi-matrix that is obtained by deleting some rows, columns, or elements of a quasi-matrix.
\end{definition}

\begin{example}
	$$
	A=\begin{bmatrix*}
	a& &b\\
	c&d& \\
	e&f&g
	\end{bmatrix*}
	$$
	is a quasi-matrix and $\begin{bmatrix*}
	a& &b\\
	&d& 
	\end{bmatrix*}$ is a subquasi-matrix of $A$.
\end{example}

\begin{definition}
	A binary quasi-matrix is a quasi-matrix having exactly two elements in each nonempty row and column.
\end{definition}

\begin{example}
	All $3\times 3$ binary quasi-matrices are listed below:
	\begin{gather*}
	\begin{bmatrix*}
	a&b&\\
	&c&d\\
	e&&f
	\end{bmatrix*},\
	\begin{bmatrix*}
	a&b&\\
	c&&d\\
	&e&f
	\end{bmatrix*},\
	\begin{bmatrix*}
	a& &b\\
	&c&d\\
	e&f&
	\end{bmatrix*},\
	\begin{bmatrix*}
	a& &b\\
	c&d& \\
	&e&f
	\end{bmatrix*},\
	\begin{bmatrix*}
	&a&b\\
	c&d& \\
	e&&f
	\end{bmatrix*},\
	\begin{bmatrix*}
	&a&b\\
	c&&d \\
	e&f&
	\end{bmatrix*}
	\end{gather*}
\end{example}

Note that a binary quasi-matrix is a square matrix, up to deleting an empty row or column. Since we usually identify a quasi-matrix canonically with the one obtained by deleting any empty row or column, in the sequel we usually consider a binary quasi-matrix as a square matrix.

\begin{definition}
	Let $A=(a_{ij})$ be an $n\times n$ binary quasi-matrix over a ring $R$. A binary quasi-determinant of $A$ is an element
	$$
	a_{1\sigma(1)}a_{2\sigma(2)}\dots a_{n\sigma(n)}-a_{1\tau(1)}a_{2\tau(2)}\dots a_{n\tau(n)}
	$$
	where $\sigma,\tau$ are permutations of $\{1,2,\dots,n\}$ such that $\sigma(l)\neq\tau(l)$ for all $1\le l\le n$. A quasi-determinant of a binary subquasi-matrix $A$ is called a binary quasi-minor of $A$.
\end{definition}

Note that by definition, if $\delta$ is a binary quasi-determinant of a quasi-matrix, then so is $-\delta$. In the sequel, we will usually consider a given binary quasi-minor up to sign.

\begin{remark}
	(1) Note that the quasi-determinant of a $2\times 2$ binary quasi-matrix is equal to its determinant, up to sign. Hence all $2\times 2$ minors (which exist) of a quasi-matrix are binary quasi-minors.
	
	(2) Note that a quasi-determinant of a $3\times 3$ binary quasi-matrix is uniquely determined up to sign. However, in general it is not equal to the determinant, even up to sign, of
	the matrix obtained by assigning value zero to all empty positions.
	
	(3) For $n\ge4$, a quasi-determinant of a binary $n\times n$ quasi-matrix is not even unique, up to sign. For example consider the following binary quasi-matrix
	$$
	\begin{bmatrix*}
	a&b& & \\
	c&d& & \\
	& &e&f\\
	& &g&h
	\end{bmatrix*}.
	$$
	Then $adeh-bcgf$ and $adgf-bceh$ are both quasi-determinants.
\end{remark}

\begin{notation}
	If $A$ is a quasi-matrix with entries in $R$, then we denote the ideal generated by the binary quasi-minors of $A$ by $I_{bin}(A)$. 
\end{notation}

\begin{example}
	Consider the quasi-matrix $A$ as below:
	$$
	A=\begin{bmatrix*}
	a&b& \\
	c&d&e\\
	f& &g
	\end{bmatrix*},
	$$
	then $adg-bef$, $bef-adg$, $ad-bc$, $bc-ad$, $cg-fe$, and $fe-cg$ are all binary quasi-minors of $A$.
\end{example}

The next elementary result shows that the ideal of binary quasi-minors generalizes to the quasi-matrices the classical ideal of $2\times 2$ minors.

\begin{proposition}\label{minor-binary-quasi-minor} 
	Let A be a matrix. Then $I_{bin}(A)=I_2(A)$.
\end{proposition}
\begin{proof}
	It is enough to show that every binary quasi-minor in $A$ is an $R$-combination of $2\times 2$ minors. Let $\delta=V_1V_2\dots V_n-W_1W_2\dots W_n$ be an arbitrary binary quasi-minor. We induct on $n\ge2$. Since the result is clear for $n=2$, we may assume $n\ge3$ and that the result holds for binary quasi-minors of size $<n$.
	
	We may assume $V_1$ is in the same row with $W_1$ and $V_2$ is in the same column with $W_1$. Let $U$ be the entry of $A$ in the same column as $V_1$ and same row as $V_2$. Then 
	\begin{gather*}
	\begin{aligned}
	&&\delta&=\delta-UW_1V_3\dots V_n+UW_1V_3\dots Vn\\
	&&      &=(V_1V_2-UW_1)V_3\dots V_n+W_1(UV_3\dots V_n-W_2\dots W_n).
	\end{aligned}
	\end{gather*}
	If $U$ is not one of the $W$'s, then the subquasi-matrix obtained by deleting the first row and column involving $W_1$ and $V_2$, and containing $U$ is binary quasi-matrix, with $UV_3\dots V_n-W_2\dots W_n$ as an $(n-1)$-sized binary quasi-minor.
	
	On the other hand, if $U$ is a $W_i$, say $W_2$ (which can only happen if $n\ge4$), then $UV_3\dots V_n-W_2\dots W_n=W_2(V_3\dots V_n-W_3\dots W_n)$ and $V_3\dots V_n-W_3\dots W_n$ is a
	binary quasi-minor of A of size $(n-2)$. In either case, we are done by induction.
\end{proof}

We say that a quasi-matrix is $\textit{generic}$ over a ring if its entries are distinct variables. 

\begin{proposition}\label{Grobner-bases-minor-1} 
Let $A$ be a generic quasi-matrix over a polynomial ring over a ring $R$ and $\underline{s}=s_1,\dots,s_k$ be a permutable weak regular sequence and none of them is unit. Then the set of binary quasi-minors of $B=(\underline{s}|A)$ is a universal Gr\"{o}bner basis for the ideal $I_{bin}(B)$.
\end{proposition}
\begin{proof}
	By Buchberger's Criterion, it is enough to show that for each pair of binary quasi-minors $f$ and $g$, the $S^{'}$-polynomial $S^{'}(f,g)$ reduces to zero modulo the set of binary quasi-minors. Let 
	$$
	f=V_1\dots V_n-W_1\dots W_n,\  \text{and}\ g=Y_1\dots Y_m-Z_1\dots Z_m.
	$$ 
	We may assume that $\lt(f)=V_1\dots V_n$ and $\lt(g)=Y_1\dots Y_m$. Then
	$$
	h=S^{'}(f,g)=-Y_1\dots Y_tW_1\dots W_n+V_1\dots V_sZ_1\dots Z_m,
	$$
	where we have reordered if necessary to assume that $Y_1,\dots,Y_t$ are exactly the $Y$'s that are not $V$'s, and $V_1,\dots,V_s$ are exactly the $V$'s that are not $Y$'s. We remark that $W_i,Y_i,Z_i,V_i$ are not necessarily distinct.
	
	We consider the subquasi-matrix $C$ of $B$ consisting of all the elements $W_i,Y_i,Z_i$ and $V_i$ that appear in $h$. If two of these elements coincide we say that the entry has multiplicity $2$ in the quasi-matrix $B$.
	
First of all, we show that each of $W_i$ and $Y_i$ in the first term of $h$ is in the same row with at least one of the $Z_i$ and $V_i$'s in the second term of $h$. Clearly every $Y_i$ is in the same row with a $Z_j$. On the other hand, for a $W_i$, it is in the same row as a $V_j$; if $j\le s$ we are done. But if $j>s$, since $V_j$ is in the same row as a $Z_k$, it follows that $W_i$ is in the same row with this $Z_k$. The same argument can be made for elements in the second term and for columns. 
	
Second, we show that in every row (column) of $C$ we have either exactly two of $W_i,Z_i,V_i,Y_i$ or exactly four of them. To verify this, it suffices to show that if three of them are in the same row (resp. column), then fourth one is also in this row (resp. column). If $W_i,Y_j,V_k$ are in the same row then some $Z_l$ is too. On the other hand, let $W_i,Y_j,Z_k$ be in the same row. If $W_i$ is in the same row (resp. column) with a $V_l$, for $l>s$, then $Y_j$ and $Y_p$ ($p>t$) are in the same row, which is not possible. Hence $W_i$ is in the same row with a $V_l$ for $l\le s$. The remaining cases follow by symmetry, and the case for columns is similar.
	
Now we see that for the binomial $h$, a row (resp. column) of $C$ contains either two factors (counting multiplicity), each appearing in separate terms, or four factors appearing in each term in two pairs. We say that $h$ and $C$ have the evenness property. 
	
If the two terms in $h$ have a common factor, say $W=W_i=Z_j$, then we consider the polynomial $h^{'}=h/W$, and the corresponding quasi-submatrix $C^{'}$ obtained by deleting $W$. Then both still have the evenness property. To produce a standard expression of $h$, it is enough to produce one for $h^{'}$. By factoring out all such common factors, we may reduce to the case that the binomial $h$ has relatively prime terms. We again denote the resulting binomial by $h$ and the corresponding quasi-submatrix by $C$.
	
Next, we associate to $h$ a (multi-)graph $H$ as follows: vertices of $H$ are the entries of $C$. If only two factors lie in a row (or column) and appear in $h$ in opposite terms, then they are joined by an edge in $H$. On the other hand, if there are four factors in a row (or column), $V,W,Y,Z$ (counting multiplicities), then we attach the vertices $W$ and $Z$, and the vertices $V$ and $Y$. This is an arbitrary choice. Then in $H$, the degree of any vertex is 2 or 4. In the graph $H$, we refer to edges as being either horizontal or vertical, depending on the positioning of the entries of the corresponding quasi-matrix.
	
In the next step, we reduce to the case that both terms in the binomial $h$ are squarefree. It is clear that the exponent of every variable in both terms of $h$ is at most 2. Suppose that $h$ has such a non-squarefree term. We choose a factor $W$ having multiplicity 2. We choose a circuit $H_1$ starting at $W$, whose edges are successively vertical and horizontal, and starts initially vertically. We consider $H_1$ as a subgraph and let $H_2$ be the subgraph by removing all edges of $H_1$ and remaining isolated vertices. The vertices of $H_1$ and $H_2$ form two subquasi-matrices of $C$ and the element $W$ appears in both of them with one other element in its row and one other element in its column. Associated to the corresponding subquasi-matrices we have associated binomials
	$$
	h^{'}=m_1-m_2,\ h^{''}=n_1-n_2,
	$$
	where $m_i$ (resp. $n_i$) is the product of the factors in $H_1$ (resp. $H_2$) in the $i$-th term of $h$. Each polynomial and its associated subquasi-matrix has the evenness property and the factor $W$ now has multiplicity one in each graph and polynomial. Moreover,
	$$
	h=m_1n_1-m_2n_2.
	$$
	We claim that $\lt(h)\ge \min\{m_1n_2,m_2n_1\}$. If not, then
	$$
	m_1n_1\le\lt(h)<m_1n_2\Rightarrow n_1<n_2
	$$
	and
	$$
	m_2n_2\le\lt(h)<m_2n_1\Rightarrow n_2<n_1
	$$
	which is a contradiction.
	
	Assume $\lt(h)\ge m_1n_2$. Then
	\begin{gather*}
	\begin{aligned}
	&&h&=m_1n_1-m_2n_2\\
	&& &=m_1n_1-m_2n_2-m_1n_2+m_1n_2\\
	&& &=m_1(n_1-n_2)+n_2(m_1-m_2)\\
	&& &=m_1h^{''}+n_2h^{'}.
	\end{aligned}
	\end{gather*}
	Further, $\lt(m_1h^{''})=m_1\lt(h^{''})=m_1\lt(n_1-n_2)$, $\lt(n_2h^{'})=n_2\lt(h^{'})=n_2\lt(m_1-m_2)$ and both are $\le\lt(h)$ by our assumption. This proves the claim.
	
	Repeating this construction for every factor of multiplicity $>1$, we eventually obtain a standard expression
	$$
	h=\sum m_ih_i\ (*)
	$$
	with $m_i$ monomial, $\lt(h)\ge\lt(m_ih_i)$ for all $i$, and $h_i$ is a binomial with relatively prime terms, all factors of multiplicity 1, and the $h_i$ and their associated subquasi-matrix has the evenness property.
	
	To complete the proof, we claim that there is a standard expression $(*)$ in which each $h_i$ is a binary quasi-minor. To accomplish this, for any binomial $b=h_i$ as above, we let $\tau(b)$ denote the number of rows and columns of the associated subquasi-matrix of $b$ having four (necessarily distinct) entries. We will show that if $\tau(b)>0$, then there is a pair $(b^{'},b^{''})$ as above, such that $\tau(b^{'})<\tau(b)$ and $\tau(b^{''})<\tau(b)$. Applying this repeatedly, we obtain a standard expression $(*)$ in which every binomial $h_i$ that appears satisfies $\tau(h_i)=0$. But a binomial $b$ with $\tau(b)=0$ is precisely a binary quasi-minor of $B$, which would prove the claim.
	
	To verify that we may decrease $\tau$ in the prescribed manner, suppose that $b=h_i$ and $\tau(b)>0$, with a row, say, with entries $W,Y,Z,V$. We again consider the graph $H$ of $b$, in this case every vertex now having degree 2. We note that the graph $H$ is not necessarily connected. Indeed, starting vertically from a vertex $W$ two cases occur. The path reaches $Z$ without passing through $V$ or $Y$, or before the path reaches $Z$ it reaches $V$ or $Y$ (actually the $Y$ cannot be reached vertically). In the second case we remove edges $WZ$ and $VY$ from $H$, and instead we add edges $VW$ and $YZ$ to $H$. In either case we have a circuit $H_1$ in which only two vertices lie in it. Again we let $H_2$ be the subgraph by removing all edges of $H_1$ and remaining isolated vertices. Now for the corresponding binomials $b^{'}$ and $b^{''}$ as above, we obtain that $\tau(b^{'})<\tau(b)$ and $\tau(b{''})<\tau(b)$. This completes the proof.
\end{proof}

\section{Equations of the multi-Rees algebra}

We fix a permutable weak regular sequence $\underline{s}=s_1,\dots,s_n$ in any ring $R$, where none of them is a unit. Let $R[I_1^{a_1}t_1,\dots,I_r^{a_r}t_r]$ be the multi-Rees algebra of powers of ideals $I_i$, where the $a_i$'s are positive integers and the ideals are generated by arbitrary subsets of $\underline{s}$. In the rest of this paper by generators of $I_i$ we mean these generators. We denote $\boldsymbol{a}=a_1,\dots,a_r$.

\begin{definition}
	Let $a$ be positive integer, $\mathscr{T}_{a}$ and $\mathscr{T}_{a}^{'}$ denote the sets in $\left(\mathbb{Z}_{\ge 0}\right)^{n-1}$ and $\mathbb{N}^{n-1}$ respectively, that are defined as follows:
	$$
	\mathscr{T}_{a}:=\{\underline{j}=(j_{n-1},...,j_1)\ | \ 0=j_0\le j_1 \le j_2 \le \dots \le j_{n-1} \le j_n=a \}
	$$
	$$
	\mathscr{T}_{a}^{'}:=\{\underline{j}=(j_{n-1},...,j_1)\ | \ 1=j_0\le j_1 \le j_2 \le \dots \le j_{n-1} \le j_n=a \}.
	$$
\end{definition}

Note that the cardinality of $\mathscr{T}_{a}$ (resp. $\mathscr{T}_{a}^{'}$) is $\binom{a+n-1}{n-1}$ (resp. $\binom{a+n-2}{n-1}$).
\begin{definition}\label{defpower}
	For $\underline{j}\in\mathscr{T}_{a}$ we define
	\begin{gather*}
	s^{\underline{j}}:=\prod_{i=1}^{n} s_{i}^{j_{i}-j_{i-1}}.
	\end{gather*}
\end{definition}

\begin{definition}
	We define the function 
	$$\underline{j}^{|k\rangle}:\mathscr{T}_{a}^{'}\rightarrow\mathscr{T}_{a}\ \text{by}\ \underline{j}^{|k\rangle}((j_{n-1},...,j_1))=(j_{n-1},...,j_k,j_{k-1}-1,...,j_{1}-1)
	$$ 
	where $1\le k\le n$. For convenience instead of $\underline{j}^{|k\rangle}((j_{n-1},...,j_1))$ we write $\underline{j}^{|k\rangle}$.
\end{definition}

\begin{definition}
	Let $a$ be a positive integer. We define $\mathscr{F}_{a}^l$ as a subset of $\mathscr{T}_{a}$ such that $\underline{j}\in\mathscr{F}_{a}^l$ if and only if $s^{\underline{j}}\in I_l^a$.
\end{definition}

Since $I_l^{a_l}$ is generated by all $s^{\underline{j}}$'s ($\underline{j}\in\mathscr{F}_{a_l}^l$), the multi-Rees algebra $R[I_1^{a_1}t_1,\dots,I_r^{a_r}t_r]$ is the ring $R\left[\{s^{\underline{j}}t_l\}_{1\le l\le r,\underline{j}\in\mathscr{F}_{a_l}^l}\right]$. Let $S:=R\left[\{T_{l,\underline{j}}\}_{1\le l\le r,\underline{j}\in\mathscr{F}_{a_l}^l}\right]$. We define and fix the $R$-algebra epimorphism 
$$
\phi : S \to R[I_1^{a_1}t_1,\dots, I_r^{a_r}t_r], \ \text{by} \ \phi(T_{l,\underline{j}})=s^{\underline{j}}t_l.
$$
We want to find generators of $\mathscr{L}=\ker(\phi)$.

\begin{definition}
	For a fixed $l$, consider the set $\{(l,\underline{j});\ \underline{j}\in\mathscr{T}_{a_l}^{'}\}$ ordered lexicographically as a subset of $\mathbb{N}^n$. We define the matrix $B_{a_l}$, whose entry in row $k$ and column $(l,\underline{j})$ is $T_{l,\underline{j}^{|k\rangle}}$.
\end{definition} 

We see that $\phi(T_{l,\underline{j}^{|k\rangle}})$ contains at least a factor of $s_k$. Let $I_l=\langle s_{k_1},\dots,s_{k_v}\rangle$, and $k_1<k_2<\dots<k_v$. Then the only possible $T_{l,\underline{j}^{|k\rangle}}$'s whose images under $\phi$ are monomials in $s_{k_1},\dots,s_{k_v}$ are in rows $k_1,\dots,k_v$ of $B_{a_l}$. On the other hand for $u<w$, if we compare images of $T_{l,\underline{j}^{|u\rangle}}$ and $T_{l,\underline{j}^{|w\rangle}}$, 
$$
\phi(T_{l,\underline{j}^{|u\rangle}})=s_1^{j_1-1-0}s_2^{j_2-j_1}\dots s_{u-1}^{j_{u-1}-j_{u-2}}s_u^{j_u-j_{u-1}+1}\dots s_w^{j_w-j_{w-1}}\dots s_n^{j_n-j_{n-1}},
$$
$$
\phi(T_{l,\underline{j}^{|w\rangle}})=s_1^{j_1-1-0}s_2^{j_2-j_1}\dots s_u^{j_u-j_{u-1}}\dots s_{w-1}^{j_{w-1}-j_{w-2}}s_w^{j_w-j_{w-1}+1}\dots s_n^{j_n-j_{n-1}},
$$
then we see that when we move on the column $(l,\underline{j})$ from row $u$ to row $w$ we lose one factor of $s_u$ and we get one factor of $s_w$. 

\begin{definition}
	In the matrix $B_{a_l}$, in the row $k_1$, we choose $T_{l,\underline{j}^{|k\rangle}}$'s whose images under $\phi$ are monomials in $s_{k_1},\dots,s_{k_v}$. We define the quasi-matrix $D_{a_l}$ to be the subquasi-matrix of $B_{a_l}$ by choosing these columns and rows $k_i$, $1\le i\le v$. The entries of the submatrix $D_{a_l}$ are all $T_{l,\underline{j}^{|k\rangle}}$'s in $B_{a_l}$ whose images under $\phi$ are monomials in $s_{k_1},\dots,s_{k_v}$. 
	
	We define the matrix $B_{\boldsymbol{a}}\coloneqq(B_{a_1}|B_{a_2}|\dots|B_{a_r})$ and the matrix $C_{\boldsymbol{a}}\coloneqq(\underline{s}|B_{\boldsymbol{a}})$. We also define the subquasi-matrices $D_{\boldsymbol{a}}\coloneqq(D_{a_1}|D_{a_2}|\dots|D_{a_r})$ and $E_{\boldsymbol{a}}\coloneqq(\underline{s}|D_{\boldsymbol{a}})$ of $C_{\boldsymbol{a}}$.
\end{definition}

The theorem 
below describes the defining equations of the multi-Rees algebra $R[I_1^{a_1}t_1,\dots,I_r^{a_r}t_r]$.

\begin{theorem}\label{main-result-1}
	Let $R$ be a Noetherian ring and suppose that ideals $I_i$ are generated by subsets of a fixed permutable weak regular sequence $s_1,\dots,s_n$, where none of them is a unit. Then
	$$
	R[I_1^{a_1}t_1,\dots,I_r^{a_r}t_r]\cong S/I_{bin}(E_{\boldsymbol{a}}).
	$$
\end{theorem}

\begin{remark}
	We will show that the defining ideal is generated by the $2\times2$ minors of $E=E_{\boldsymbol{a}}$ involving $s_1,\dots,s_n$, the $2\times2$ minors of the $D_{a_l}$, and the binary quasi-minors of $E$ (which are not minors) and have at most two entries from each $D_{a_l}$.
\end{remark}

\begin{proof}
	First we show that $I_{bin}(E_{\boldsymbol{a}})\subseteq \mathscr{L}$. If $f$ is a $T$-binary quasi-minor, then 
	$$
	f=T_{l_1,{\underline{j}^{(1)}}^{|k_{i_1}\rangle}}T_{l_2,{\underline{j}^{(2)}}^{|k_{i_2}\rangle}}\dots T_{l_\beta,{\underline{j}^{(\beta)}}^{|k_{i_\beta}\rangle}}-T_{l_1,{\underline{j}^{(1)}}^{|k_{u_1}\rangle}}T_{l_2,{\underline{j}^{(2)}}^{|k_{u_2}\rangle}}\dots T_{l_\beta,{\underline{j}^{(\beta)}}^{|k_{u_\beta}\rangle}}.
	$$
	For arbitrary $T_{l_v,{\underline{j}^{(v)}}^{|k_{i_v}\rangle}}$ we have
	$$
	\phi(T_{l_v,{\underline{j}^{(v)}}^{|k_{i_v}\rangle}})=t_{l_v}s_{k_{i_v}}\prod_{\alpha=1}^{n}s_\alpha^{j_\alpha^{(v)}-j_{\alpha-1}^{(v)}},\ \text{with}\ \underline{j}^{(v)}\in \mathscr{T}_{a_{l_v}}^{'},
	$$
	and similarly
	$$
	\phi(T_{l_v,{\underline{j}^{(v)}}^{|k_{u_v}\rangle}})=t_{l_v}s_{k_{u_v}}\prod_{\alpha=1}^{n}s_\alpha^{j_\alpha^{(v)}-j_{\alpha-1}^{(v)}},\ \text{with}\ \underline{j}^{(v)}\in \mathscr{T}_{a_{l_v}}^{'}.
	$$
	Hence we have
	\begin{gather*}
	\phi(f)=s_{k_{i_1}}s_{k_{i_2}}\dots s_{k_{i_\beta}}t_{l_1}t_{l_2}\dots t_{l_\beta}\prod_{\alpha=1}^{n}s_\alpha^{j_\alpha^{(1)}-j_{\alpha-1}^{(1)}}\prod_{\alpha=1}^{n}s_\alpha^{j_\alpha^{(2)}-j_{\alpha-1}^{(2)}}\dots\prod_{\alpha=1}^{n}s_\alpha^{j_\alpha^{(\beta)}-j_{\alpha-1}^{(\beta)}}\\
	-s_{k_{u_1}}s_{k_{u_2}}\dots s_{k_{u_\beta}}t_{l_1}t_{l_2}\dots t_{l_\beta}\prod_{\alpha=1}^{n}s_\alpha^{j_\alpha^{(1)}-j_{\alpha-1}^{(1)}}\prod_{\alpha=1}^{n}s_\alpha^{j_\alpha^{(2)}-j_{\alpha-1}^{(2)}}\dots\prod_{\alpha=1}^{n}s_\alpha^{j_\alpha^{(\beta)}-j_{\alpha-1}^{(\beta)}}=0.
	\end{gather*}
	If $f$ is an $\underline{s}$-binary quasi-minor, then 
	$$
	f=s_{k_{i_1}}T_{l_2,{\underline{j}^{(2)}}^{|k_{i_2}\rangle}}\dots T_{l_\beta,{\underline{j}^{(\beta)}}^{|k_{i_\beta}\rangle}}-s_{k_{u_1}}T_{l_2,{\underline{j}^{(2)}}^{|k_{u_2}\rangle}}\dots T_{l_\beta,{\underline{j}^{(\beta)}}^{|k_{u_\beta}\rangle}},
	$$
	and similarly we see that $\phi(f)=0$:
	\begin{gather*}
	\phi(f)=s_{k_{i_1}}s_{k_{i_2}}\dots s_{k_{i_\beta}}t_{l_2}\dots t_{l_\beta}\prod_{\alpha=1}^{n}s_\alpha^{j_\alpha^{(2)}-j_{\alpha-1}^{(2)}}\dots\prod_{\alpha=1}^{n}s_\alpha^{j_\alpha^{(\beta)}-j_{\alpha-1}^{(\beta)}}\\
	-s_{k_{u_1}}s_{k_{u_2}}\dots s_{k_{u_\beta}}t_{l_2}\dots t_{l_\beta}\prod_{\alpha=1}^{n}s_\alpha^{j_\alpha^{(2)}-j_{\alpha-1}^{(2)}}\dots\prod_{\alpha=1}^{n}s_\alpha^{j_\alpha^{(\beta)}-j_{\alpha-1}^{(\beta)}}=0.
	\end{gather*}
	Thus $I_{bin}(E_{\boldsymbol{a}})\subseteq \mathscr{L}$. Now we prove that $\mathscr{L}\subseteq I_{bin}(E_{\boldsymbol{a}})$. We prove this claim in two steps. First we prove the claim for the case that $\boldsymbol{a}=(1,\dots,1)$.
	
	We consider the ideal $I=\langle s_1\dots,s_n\rangle$. Then we have the following commutative diagram:
	$$
	\begin{tikzcd}
	R\left[\{T_{l,\underline{j}}\}_{1\le l\le r,\underline{j}\in\mathscr{F}_{1}^l}\right]\arrow[hook]{d}{\theta_1}\arrow[two heads]{r}{\phi}                 &R[I_1t_1,\dots,I_rt_r]\arrow[hook]{d}{\theta_2}\\
	R\left[\{T_{l,\underline{j}}\}_{1\le l\le r,\underline{j}\in\mathscr{T}_{1}}\right]\arrow[two heads]{r}{\phi^{'}}&R[It_1,\dots,It_r],
	\end{tikzcd}
	$$ 
	where $\phi^{'}(T_{l,\underline{j}})=s^{\underline{j}}t_l$. For any $f\in\mathscr{L}$, we have
	$$
	\phi^{'}(f)=\phi^{'}\theta_1(f)=\theta_2\phi(f)=0\Rightarrow f\in\ker(\phi^{'}).
	$$
If $\underline{s}=s_1,\dots,s_n$ is a permutable regular sequence, then by \cite[Proposition 3.1]{ribbe1999defining}, and \cite[Corollary 5.5.5]{huneke2006integral}, $f\in I_2(C_{\boldsymbol{a}})$. In the case that $\underline{s}=s_1,\dots,s_n$ is not a regular sequence, clearly $I=R$ is an ideal of linear type, then by Corollary~\ref{syzygy} and \cite[Proposition 3.1]{ribbe1999defining} $f\in I_2(C_{\boldsymbol{a}})$. By Proposition~\ref{minor-binary-quasi-minor}, $f\in I_{bin}(C_{\boldsymbol{a}})$. 
	
We put a total order on variables of $C_{\boldsymbol{a}}$ such that all variables which are not in $E_{\boldsymbol{a}}$ are greater than all variables which are in $E_{\boldsymbol{a}}$. We also put the lexicographic order on $R\left[\{T_{l,\underline{j}}\}_{1\le l\le r,\underline{j}\in\mathscr{T}_{1}}\right]$. By Proposition~\ref{Grobner-bases-minor-1} the set of all binary quasi-minors of $C_{\boldsymbol{a}}$ form a Gr\"{o}bner basis for $I_{bin}(C_{\boldsymbol{a}})$, and so $\lt(f)$ can be generated by leading terms of binary quasi-minors of $C_{\boldsymbol{a}}$. Hence we have an expression $\lt(f)=\sum r_if_i\lt(g_i)$, where $g_i$'s are binary quasi-minors, $r_i\in R$, and $f_i$'s are monomials. We take $h_1=\sum r_if_ig_i$, and we see that $\lt(h_1)=\lt(f)$. On the other hand the variables in $\lt(g_i)$ are some of the variables in $f$ and so they are variables in $E_{\boldsymbol{a}}$. Since we have the lexicographic order and all variables out of $E_{\boldsymbol{a}}$ are greater than these variables, variables in non-leading terms of $g_i$ are also in $E_{\boldsymbol{a}}$. It is also clear that variables in $f_i$ are some of variables of $f$. Hence $h_1$ is a polynomial with variables in $E_{\boldsymbol{a}}$ and $g_i$'s are binary quasi-minors in $E_{\boldsymbol{a}}$. Therefore $h_1\in I_{bin}(E_{\boldsymbol{a}})$, $f-h_1\in I_{bin}(C_{\boldsymbol{a}})$, and variables in $f-h_1$ are in $E_{\boldsymbol{a}}$. Moreover $\lt(f)>_{lex} \lt(f-h_1)$. We continue this procedure until $\lt(f-h_1-h_2-\dots-h_m)=0$. Hence $f=h_1+h_2+\dots+h_m$ and so $f\in I_{bin}(E_{\boldsymbol{a}})$.
	
	Now we prove the Theorem for every positive $a_i$. The proof in this step is similar to \cite{jabarnejad2016rees}, and here we only outline it. 
	
	Let $Y=(x_{i,l})$ be a generic $n\times r$ matrix and let $X$ be a subquasi-matrix of $Y$ defined as follows: an arbitrary $x_{i,l}$ is an entry of $X$ if $s_i$ is between generators of $I_l$. We define $x_{l}^{\underline{j}}$ to be the element $s^{\underline{j}}$, where we replace $s_i$'s with the variables in the $l$-column of $Y$. Let $A_l$ be the $a_l$-th Veronese subring $R\left[\{x_{i,l}\}_{x_{i,l}\ \text{is in the}\ l\text{-column of}\ X} \right]^{(a_l)}$. We also define $X_l(a)$ to be the family of monomials of degree $a$ ($a\in \mathbb{N}$) in the variables that are in the $l$-column of $X$. If $A$ is the $\boldsymbol{a}$-th Veronese subring $R[X]^{\boldsymbol{a}}=R[X_1(a_1),X_2(a_2),...,X_r(a_r)]$ and $\alpha : S\rightarrow A$, then we see that
	$$
	\ker(\alpha)=\langle I_2(D_{a_1}),I_2(D_{a_2}), ... ,I_2(D_{a_r})\rangle\subseteq I_{bin}(E_{\boldsymbol{a}}).
	$$
	
	We define the quasi-matrix $Z=(\underline{s}|X)$. We define the map 
	$$
	\psi : R[X] \rightarrow R[I_1u_1,I_2u_2,\dots,I_ru_r], \quad \psi(x_{i,l})=s_iu_l
	$$
	then by the first part of the proof $\ker(\psi)$ is generated by $2\times 2$ $\underline{s}$-minors and $x$-binary minors (binary minors which don't contain $s_i$) of $Z$.
	
	We have similar digram to \cite{jabarnejad2016rees}, and we can define the map $g: A\rightarrow R[I_1^{a_1}t_1,\dots,I_r^{a_r}t_r]$. We see that the kernel of $g$ is generated by polynomials of the form 
	$$
	s_\delta x_{\gamma,l}x_l^{\underline{j}}-s_\gamma x_{\delta,l}x_l^{\underline{j}},\ \underline{j}\in \mathscr{F}_{a_l-1}^l
	$$
	and
	\begin{gather*}
	\prod_{i=1}^{v}x_{\delta_i,\gamma_i}x^{\underline{j}^{(i)}}_{\gamma_i}-\prod_{i=1}^{v}x_{\beta_i,\gamma_i}x^{\underline{j}^{(i)}}_{\gamma_i}, \ \underline{j}^{(i)}\in \mathscr{F}_{a_{\gamma_i-1}^{\gamma_i}}, 
	\end{gather*}
    where $\gamma_i$ (resp. $\delta_i$, $\beta_i$) are distinct and $\{\delta_1,\dots,\delta_v\}=\{\beta_1,\dots\beta_v\}$.
    
    On the other hand if $\underline{j}^{'}=\underline{j}+(1,\dots,1)$, then $s_\delta T_{l,{\underline{j}^{'}}^{|\gamma\rangle}}-s_\gamma T_{l,{\underline{j}^{'}}^{|\delta\rangle}}$ is an $s$-minor of $E_{\boldsymbol{a}}$ and 
    $$
    \alpha\left(s_\delta T_{l,{\underline{j}^{'}}^{|\gamma\rangle}}-s_\gamma T_{l,{\underline{j}^{'}}^{|\delta\rangle}}\right)=s_\delta x_{\gamma,l}x_l^{\underline{j}}-s_\gamma x_{\delta,l}x_l^{\underline{j}}
    $$
    and if $\underline{i}^{(k)}=\underline{j}^{(k)}+(1,\dots,1)$, then 
    $$
    T_{\gamma_1,{\underline{i}^{(1)}}^{|\delta_1\rangle}}\dots T_{\gamma_v,{\underline{i}^{(v)}}^{|\delta_v\rangle}}-T_{\gamma_1,{\underline{i}^{(1)}}^{|\beta_1\rangle}}\dots T_{\gamma_v,{\underline{i}^{(v)}}^{|\beta_v\rangle}}
    $$
    is a binary quasi-minor of $E_{\boldsymbol{a}}$ and its image under $\alpha$ is
    $$
    \left(\prod_{i=1}^{v}x_{\delta_i,\gamma_i}x^{\underline{j}^{(i)}}_{\gamma_i}-\prod_{i=1}^{v}x_{\beta_i,\gamma_i}x^{\underline{j}^{(i)}}_{\gamma_i}\right).
    $$ 
    This completes the proof.
	\end{proof}

\begin{remark}
	We prove that every $\underline{s}$-binary quasi-minor of $E_{\boldsymbol{a}}$ can be generated by $2\times 2$ $\underline{s}$-minors and $T$-binary quasi-minors of $E_{\boldsymbol{a}}$. Let $f=s_iV_1V_2\dots V_n-s_jW_1W_2\dots W_n$ ($V_i$ and $W_i$ are equal to $T_{l,\underline{j}^{|k\rangle}}$'s). Without loss of generality we may assume $s_i$ and $W_1$ are in the same row and $W_1$ and $V_1$ are in the same column. If $V_1$ and $s_j$ are in the same row, then we have 
	\begin{gather*}
	f=s_iV_1V_2\dots V_n-s_jW_1W_2\dots W_n-s_jW_1V_2\dots V_n+s_jW_1V_2\dots V_n\\
	=(s_iV_1-s_jW_1)V_2\dots V_n+s_jW_1(V_2\dots V_n-W_2\dots W_n).
	\end{gather*}
	
	If $V_1$ and $s_j$ are not in the same row, then there is an $s_k$ which is in the same row with $V_1$. We have
	\begin{gather*}
	f=s_iV_1V_2\dots V_n-s_jW_1W_2\dots W_n-s_kW_1V_2\dots V_n+s_kW_1V_2\dots V_n\\
	=(s_iV_1-s_kW_1)V_2\dots V_n+W_1(s_kV_2\dots V_n-s_jW_2\dots W_n).
	\end{gather*}
	We can continue this procedure until all generators are either $2\times 2$ $\underline{s}$-minors or $T$-binary quasi-minors of $E_{\boldsymbol{a}}$ .
\end{remark}

\begin{remark}
We can define $D_{a_l}$ in another way. If $I_l=\langle s_{i_1},\dots,s_{i_u}\rangle$, then we define similar $\mathscr{T}_{a_l}$ and $\mathscr{T}^{'}_{a_l}$ with members $\underline{j}=(j_{u-1},\dots,j_1)$. We also define $s^{\underline{j}}$ which is a monomial in $s_{i_1},\dots,s_{i_u}$. Hence we define $T_{l,\underline{j}}$ and so $D_{a_l}$ and finally $E_{\boldsymbol{a}}$. We can see that the image of corresponding entries in both $D_{a_l}$ is the same. We will have a new $R[T_{l,\underline{j}}]$ and a new $\phi$, but this new one is isomorphic to the previous one and the image of corresponding variables under different $\phi$'s is the same. Then we have similar result about the generators of the kernel.
\end{remark}

\begin{theorem}\label{main-result}
Let $R=k[x_1,\dots,x_m]$ ($k$ a field) and let $\underline{s}=s_1,\dots,s_n$ be a regular sequence of squarefree monomials. If the ideals $I_i$ ($1\le i\le r$) are generated by arbitrary subsets of $\underline{s}$ and $a_i$ ($1\le i\le r$) are positive integers, then the multi-Rees algebra $R[I_1^{a_1}t_1,\dots,I_r^{a_r}t_r]$ is a normal Cohen-Macaulay domain. 
\end{theorem}
\begin{proof}
	First we show that $R[I_1u_1,\dots,I_ru_r]$ is normal. By Theorem~\ref{main-result-1} we have
	$$
	R[I_1u_1,\dots,I_ru_r]\cong R[T_{l,\underline{j}}]/I_{bin}(E_{\boldsymbol{a}}),
	$$
	and by similar argument to Proposition~\ref{Grobner-bases-minor-1} $\lt(I_{bin}(E_{\boldsymbol{a}}))$ is generated by leading terms of binary quasi-minors which are squarefree, hence by \cite[Proposition 13.5, Proposition 13.15]{sturmfels1996grobner} $R[I_1u_1,\dots,I_ru_r]$ is a normal domain. On the other hand $R[I_1^{a_1}t_1,\dots,I_r^{a_r}t_r]\cong R[I_1^{a_1}u_1^{a_1},\dots,I_r^{a_r}u_r^{a_r}]$ and $R[I_1^{a_1}u_1^{a_1},\dots,I_r^{a_r}u_r^{a_r}]$ is a direct summand of $R[I_1u_1,\dots,I_ru_r]$, so $R[I_1^{a_1}t_1,\dots,I_r^{a_r}t_r]$ is normal.
	
	For Cohen-Macaulayness, if we consider the semigroup $M$ generated by $x_1,\dots,x_m$ and $ft_i$, where $f$ is a generator of $I_i^{a_i}$, then by \cite[Proposition 1]{hochster1972rings}, $M$ is normal, hence by \cite[Theorem 1]{hochster1972rings}, $k[M]\cong R[I_1^{a_1}t_1,\dots,I_r^{a_r}t_r]$ is a Cohen-Macaulay domain.
\end{proof}

\begin{theorem}
	Let $T=R[x_1,\dots,x_m]$ ($R$ a Noetherian ring) and let $\underline{s}=s_1,\dots,s_n$ be a regular sequence of squarefree monomials. If the ideals $I_i$ ($1\le i\le r$) are generated by arbitrary subsets of $\underline{s}$ and $a_i$ ($1\le i\le r$) are positive integers, then we have the following
	
	i. If $R$ is a Cohen-Macaulay ring, then the multi-Rees algebra $T[I_1^{a_1}t_1,\dots,I_r^{a_r}t_r]$ is a Cohen-Macaulay ring.
	
	ii. If $R$ is a normal domain, then the multi-Rees algebra $T[I_1^{a_1}t_1,\dots,I_r^{a_r}t_r]$ is a normal domain. 
\end{theorem}
\begin{proof}
	If $M$ is the semigroup mentioned in the proof of Theorem~\ref{main-result}, then $M$ is normal and hence:
	
	i. By \cite[Theorem 1]{hochster1972rings}, $R[M]\cong T[I_1^{a_1}t_1,\dots,I_r^{a_r}t_r]$ is Cohen-Macaulay.
	
	ii. This follows from \cite[Proposition 1]{hochster1972rings}.
\end{proof}

\begin{example}
	Let $R[x,y]$, where $R$ is a UFD and is not a PID. Suppose $p_1,p_2$ are two non-associate irreducible elements in $R$. Let $I_1=\langle p_1,p_2\rangle$, $I_2=\langle p_1,x\rangle$, $I_3=\langle p_2,x\rangle$, $I_4=\langle p_1,y\rangle$, $I_5=\langle p_2,y\rangle$. Then we have the homomorphism 
	\begin{gather*}
	\phi: R[T_{1,1,1,1},T_{1,1,1,0},T_{2,1,1,1},T_{2,1,0,0},T_{3,1,1,0},T_{3,1,0,0},T_{4,1,1,1},T_{4,0,0,0},T_{5,1,1,0},T_{5,0,0,0}]\\
	\rightarrow R[I_1t_1,I_2t_2,I_3t_3,I_4t_4,I_5t_5],\\
	T_{1,1,1,1}\mapsto p_1t_1,\ T_{1,1,1,0}\mapsto p_2t_1,\ T_{2,1,1,1}\mapsto p_1t_2,\ T_{2,1,0,0}\mapsto xt_2,\ T_{3,1,1,0}\mapsto p_2t_3,\\ 
	T_{3,1,0,0}\mapsto xt_3,\ T_{4,1,1,1}\mapsto p_1t_4,\ T_{4,0,0,0}\mapsto yt_4,\ T_{5,1,1,0}\mapsto p_2t_5,\ T_{5,0,0,0}\mapsto yt_5,   
	\end{gather*}
	and its kernel is 
	\begin{gather*}
	\langle p_1T_{1,1,1,0}-p_2T_{1,1,1,1},p_1T_{2,1,0,0}-xT_{2,1,1,1},p_2T_{3,1,0,0}-xT_{3,1,1,0},p_1T_{4,0,0,0}-yT_{4,1,1,1},\\
	p_2T_{5,0,0,0}-yT_{5,1,1,0}, T_{1,1,1,1}T_{3,1,1,0}T_{2,1,0,0}-T_{1,1,1,0}T_{2,1,1,1}T_{3,1,0,0},\\T_{1,1,1,1}T_{5,1,1,0}T_{4,0,0,0}-T_{1,1,1,0}T_{4,1,1,1}T_{5,0,0,0},
	T_{2,1,1,1}T_{3,1,0,0}T_{4,0,0,0}T_{5,1,1,0}-T_{2,1,0,0}T_{3,1,1,0}T_{4,1,1,1}T_{5,0,0,0}\rangle.
	\end{gather*}
	We see that $E_{\boldsymbol{a}}$ has the form below
	$$
	\begin{bmatrix*}
	p_1&T_{1,1,1,1}&T_{2,1,1,1}&           &T_{4,1,1,1}&           \\
	p_2&T_{1,1,1,0}&           &T_{3,1,1,0}&           &T_{5,1,1,0}\\
	x&           &T_{2,1,0,0}&T_{3,1,0,0}&           &           \\
	y&           &           &           &T_{4,0,0,0}&T_{5,0,0,0}
	\end{bmatrix*}.
	$$
\end{example}

\section{Acknowledgment}

I would like to express my deepest gratitude to my advisor, Prof. Mark Johnson, for his support and guidance throughout this work.

\end{document}